\documentclass[12pt]{article}
\usepackage{ulem}
\usepackage{graphicx,amsmath,amsfonts,amssymb,color,mathrsfs, amsmath,amsthm}
\usepackage{epsfig}
\usepackage{subfigure}
\usepackage{bm}

\newcommand{\chuhao}{\fontsize{19pt}{\baselineskip}\selectfont}

\numberwithin{equation}{section}

 \newtheorem{theorem}{Theorem}[section]
 \newtheorem{lemma}{Lemma}[section]

 \newtheorem{remark}{Remark}[section]

 \newtheorem{example}{Example}[section]
 
 \newtheorem{corollary}{Corollary}[section]
 \newtheorem{definition}{Definition}[section]

\title{\bf\color{black} \chuhao{Hoeffding's inequality for continuous-time Markov chains}}

\author{Jinpeng Liu
\thanks{School of Mathematics and Statistics, Xidian University, Xi'an, Shanxi, 710071, P.R. China, E-mail: liujinpeng123@xidian.edu.cn.
\newline \indent \hspace*{-2.5mm}
$^{**}$ School of Mathematics and Statistics, HNP-LAMA, Central South University, Changsha, Hunan, 410083, P.R. China, E-mail: liuyy@csu.edu.cn.
\newline \indent \hspace*{-2.5mm}
$^{***}$ School of Information Science and Engineering, Hunan Women's University, Changsha, Hunan, 410004, P.R. China, E-mail: zhoulin@hnwu.edu.cn.
}\ \ , Yuanyuan Liu$^{**}$\ \ and\ \ Lin Zhou$^{***}$}

\date{\today}

\begin{document}
 \maketitle
\maketitle
\begin{abstract}
Hoeffding's inequality is a fundamental tool widely applied in probability theory, statistics, and machine learning.
In this paper, we establish Hoeffding's inequalities specifically tailored for an irreducible and positive recurrent continuous-time Markov chain (CTMC) on a countable state space with the invariant probability distribution ${\pi}$ and an $\mathcal{L}^{2}(\pi)$-spectral gap ${\lambda}(Q)$.
More precisely, for a function $g:E\to [a,b]$ with a mean $\pi(g)$, and given $t,\varepsilon>0$, we derive the inequality
\[
\mathbb{P}_{\pi}\left(\frac{1}{t} \int_{0}^{t} g\left(X_{s}\right)\mathrm{d}s-\pi (g) \geq \varepsilon \right) \leq \exp\left\{-\frac{{\lambda}(Q)t\varepsilon^2}{(b-a)^2} \right\},
\]
which can be viewed as a generalization of Hoeffding's inequality for discrete-time Markov chains (DTMCs) presented in 
[\textit{J. Fan et al.,  J. Mach. Learn. Res., 22(2022), pp. 6185-6219}] to the realm of CTMCs. 
The key analysis enabling the attainment of this inequality lies in the utilization of the techniques of skeleton chains and augmented truncation approximations.
Furthermore, we also discuss Hoeffding's inequality for a  jump process on a general state space.

\end{abstract}

\section{Introduction}

Classical Hoeffding's inequality \cite{Hoeffding1963Probability} provides a specific upper bound on the deviation between the empirical
mean of a series of independent and identically distributed (i.i.d.) random variables and their actual expected value.
To be specific, for i.i.d. random variables $\{X_i$, $i\geq0\}$ such that $\mathbb{P}(a\leq{X_i}\leq{b})=1$, where $a$ and $b$ are a pair of real numbers,
Hoeffding's inequality gives us
\[\mathbb{P}\left(\frac{1}{n}\sum_{i=0}^{n-1}X_i-\mathbb{E}[X_0]\geq\varepsilon\right)\leq\exp\left\{\frac{-2n\varepsilon^2}{(b-a)^2}\right\},\]
where $\varepsilon$ is a positive real number.
It can be seen that Hoeffding's inequality guarantees that the probability of the large deviation between the sample mean and the actual expected value will decrease exponentially with the increase of the sample size. Based on this, Hoeffding's inequality has been widely used in various fields, especially in probability limit theory, statistical learning theory, machine learning and information theory;
see, \cite{Devroye1996,Bai&Lin2009,Boucheron2013} and  their references therein.

Drawing inspiration from the diverse applications of Hoeffding's inequality, researchers have developed various generalizations of this powerful tool.
One fundamental generalization involves extending Hoeffding's inequality to Markov processes, thereby removing the assumption of independence among random variables.
This extension allows for the application of Hoeffding's inequality in a range of fields, including Markov chain Monte Carlo (MCMC) algorithms, time series analysis, and multi-armed bandit problems with Markovian rewards; see, e.g., \cite{WOS:000178279400005,2006Hidden,2010Optimal,bubeck2012regret,fan2021hoeffding}.

There is a wealth of literature available on Hoeffding's inequality for DTMCs.
For example, Glynn and Ormoneit \cite{glynn2002hoeffding} employed the minorization and drift conditions to establish a Hoeffding inequality for uniformly ergodic DTMCs.
In the same setting, Boucher \cite{boucher2009hoeffding} utilized the Drazin inverse and obtained similar results to those in the aforementioned study.
More recently, Liu and Liu \cite{liu&liu2021} derived Hoeffding's inequality for DTMCs via the solution of Poisson's equation, thereby eliminating the need for the aperiodic assumption made in the previous work \cite{glynn2002hoeffding,boucher2009hoeffding}.
On a different note, Dedecker and Gou{\"e}zel \cite{Dedecker2014SubgaussianCI} employed the coupling technique to derive a Hoeffding inequality specifically for geometrically ergodic DTMCs.
Building upon this work, Wintenberger \cite{Wintenberger2017} extended the results from Dedecker and Gou{\"e}zel \cite{Dedecker2014SubgaussianCI} to encompass unbounded functions.
Additionally, Paulin \cite{2015Concentration} developed Marton coupling techniques to derive McDiarmid's inequality for DTMCs with a finite mixing time,
which serves as a generalization of Hoeffding's inequality by replacing the sum of random variables with a function of  random variables.

The spectral gap is another crucial tool for studying Hoeffding's inequality for Markov processes.
Gillman \cite{Gillman1993} was the first to utilize the spectral gap to establish a Hoeffding inequality for a reversible DTMC on a finite state space.
Building upon this initial work, Dinwoodie \cite{Dinwoodie1995} and L{\'e}zaud \cite{1998Chernoff} further refined the Hoeffding bound in reference \cite{Gillman1993} using different techniques.
As a significant contribution, Le{\'o}n and Perron \cite{a2004optimal} derived an optimal Hoeffding inequality via the spectral gap.
Miasojedow \cite{miasojedow2014hoeffding} then extended the result of Le{\'o}n and Perron \cite{a2004optimal} to geometrically ergodic DTMCs on a general state space, eliminating the requirement of invertibility.
More recently, Fan et al. \cite{fan2021hoeffding} developed a time-dependent functional Hoeffding inequality based on the work of Le{\'o}n and Perron \cite{a2004optimal} and Miasojedow \cite{miasojedow2014hoeffding}.
Additionally, there are other Hoeffding-type inequalities constructed using methods such as Stein's method, regeneration techniques, information-theoretical ideas, and others,
as detailed in reference \cite{2006Stein,2012Exponential,2020Transport}.

In the realm of MCMC algorithms, it is an increasing need to deal with cases involving CTMCs,
and there is empirical evidence that these continuous-time MCMC algorithms are more efficient than their discrete-time counterparts; see, e.g., \cite{bouchard2018bouncy,fearnhead2018piecewise,bierkens2018piecewise}.
Moreover, additional algorithms, including diffusion Metropolis-Hastings type algorithms \cite{roberts1996geometric} and time-invariant estimating equations \cite{baddeley2000time},
also rely on the use of concentration inequalities in the context of CTMCs to control approximation errors.
Consequently, there is a compelling need to develop Hoeffding's inequality specifically for CTMCs.
Choi and Li \cite{choi2019hoeffding} built upon the techniques introduced in the works \cite{glynn2002hoeffding} and \cite{boucher2009hoeffding}
to derive a Hoeffding inequality for uniformly ergodic diffusion processes.
Liu and Liu \cite{liu&liu2021}, on the other hand, extended the results of \cite{choi2019hoeffding} to encompass general continuous-time Markov processes by employing the solution of Poisson's equation.
Additionally, there are other concentration inequalities available for CTMCs; see, e.g., \cite{1998Chernoff,lezaud2001chernoff,guillin2009transportation}.

However, the body of research on Hoeffding's inequality for CTMCs is still relatively sparse compared to the extensive results available for DTMCs.
In this paper, we will utilize the spectral gap to describe Hoeffding's inequality for CTMCs,
which serves as a parallel result to the discrete-time case presented in Fan et al. \cite{fan2021hoeffding},
but the two inequalities display subtle difference in the presentation of its upper bounds, as shown in Theorems \ref{main result} and \ref{main result dtmc}.
Since Theorem \ref{main result dtmc} characterizes the optimal Hoeffding's inequality for DTMCs using the spectral gap, see, e.g., \cite{a2004optimal,fan2021hoeffding},
this implies that our result represents the optimal Hoeffding's inequality for continuous-time case.
A comparison with existing results can be found in Remark \ref{remark-1}.
Meanwhile, we have to note that there is an additional assumption that $\bar{Q}$ is regular in Theorems \ref{main result},
which result from essential difference between CTMCs and DTMCs: for a CTMC with an unbounded $Q$-martix $Q$,
the regularity assumption on $Q$ is usually needed to guarantee the uniqueness of  the $Q$-process.

The remainder of this paper is organized as follows.
In Section \ref{sec 2}, we provide the necessary background and introduce our version of Hoeffding's inequality.
Section \ref{sec 3} presents the proof of our main result, utilizing two key tools: the skeleton chain and the collapsed chain.
Finally, in Section \ref{sec 4}, we discuss Hoeffding's inequlity for jump processes on general state spaces.

\section{Main Results}\label{sec 2}

As preparation for presenting theorems, we introduce some notations as follows.
Let $\{X_t,t\geq0\}$ be an irreducible and time-homogeneous CTMC on a countable state space $E$, and let $P_t=(P_t(i,j))_{i,j\in E}$ be the transition function of $X_t$.
We focus on the case that $X_t$ is a $Q$-process with a totally stable and regular $Q$-matrix $Q=(Q(i,j))_{i,j\in E}$.
Assume that $X_t$ is positive recurrent with the unique invariant probability distribution ${\pi}$.
Furthermore, we  assume that  the state space $E$ is equipped with a {$\sigma$}-algebra $\mathcal{B}$ and a $\sigma$-finite measure $\pi$ on $(E,\mathcal{B})$.
Let  $\mathcal{L}^{2}(E, \mathcal{B},\pi):=\{f : \pi(f^2)<\infty\}$ be the set of all square-integrable functions on $(E, \mathcal{B}, \pi)$.
For convenience, we write $\mathcal{L}^{2}(E, \mathcal{B},\pi)$ as $\mathcal{L}^{2}(\pi)$.
It is known that $\mathcal{L}^{2}(\pi)$ is the Hilbert space equipped with the following inner product and norm:
\[
	\left\langle f, h\right\rangle_{\pi}:=\pi(fh)=\sum_{i\in E}\pi(i)f(i)h(i), \ \ \forall f, h \in \mathcal{L}^{2}(\pi),
\]
\[
\|f\|_{\pi,2}:=\sqrt{\left\langle f, f\right\rangle}_{\pi}=(\pi(f^2))^{\frac{1}{2}},\ \  \forall f \in \mathcal{L}^{2}(\pi).
\]

For a transition function $P_t$ with the invariant measure $\pi$, it can be extended uniquely to $\mathcal{L}^{2}(\pi)$ as a strongly continuous contraction semigroup,
which acts to the right on functions, i.e.,
\[P_tf(i)=\sum_{j\in E}P_t(i,j)f(j), \ \ \forall i \in E,\ \forall f \in \mathcal{L}^{2}(\pi).\]
According to the ordinary semigroup theory, $P_t$ deduces an infinitesimal generator, which is denoted by $\mathcal{A}$. That is,
for any $f \in \mathcal{L}^{2}(\pi)$, if we have
\begin{equation}\label{infinitestimal}
\lim_{t\rightarrow 0^{+}}\frac{P_tf-f}{t}=g,
\end{equation}
and $g \in \mathcal{L}^{2}(\pi)$, then we define $\mathcal{A}f=g$.
Such functions $f$ consist of the domain of $\mathcal{A}$, denoted by $\mathcal{D}(\mathcal{A})$.

In Chapter 9 of \cite{Chen2004}, the $\mathcal{L}^{2}(\pi)$-spectral gap of $X_t$, denoted as $\lambda(Q)$ for convenience, is defined as follows.
\begin{definition}\label{spectral-gap}
The chain $X_t$ is said to admit an $\mathcal{L}^{2}(\pi)$-spectral gap $\lambda(Q)$ if
\[
		\lambda(Q):=\inf \left\{-\langle	{\mathcal{A}} f,f \rangle_{\pi}: f\in \mathcal{D}({\mathcal{A}}),\ \|f\|_{\pi,2}=1,\ \pi(f)=0 \right\}>0.
\]
\end{definition}

We call a CTMC $X_t$ is (time) reversible if $P_{t}$ satisfies
\[
  \pi(i)P_t(i,j)=\pi(j)P_t(j,i),  \ \ \forall i,j\in E, \ \ t\geq 0,
\]
which is equivalent to
\[
  \pi(i)Q(i,j)=\pi(j)Q(j,i),  \ \ \forall i,j\in E.
\]

In the case of that $X_{t}$ is reversible and the state space $E$ is finite, we have
\begin{equation}\label{eigenvalue}
\lambda(Q)=\lambda_1,
\end{equation}
where $\lambda_i$ is the $i\text{th}$ smallest non-trivial eigenvalue of $-Q$.
For the irreversible case, we have the following symmetrizing procedure.
Let $\hat{P}_t=(\hat{P}_{t}(i,j))_{i,j\in E}$ be the dual of $P_t$, i.e.,
\[\hat{P}_{t}(i,j)=\frac{\pi(j)P_{t}(j,i)}{\pi(i)},\ \ \forall i,j\in E.\]
It deduces the adjoint operator $\hat{\mathcal{A}}$ of $\mathcal{A}$ and the dual $Q$-matrix $\hat{Q}=\big(\hat{Q}(i,j)\big)_{i,j\in E}$, given by
\begin{equation}\label{symmetrizing procedure-1}
	    \hat{Q}(i,j)=\frac{\pi(j)Q(j,i)}{\pi(i)}, \ \ \forall i,j\in E.
\end{equation}
It is easy to check that $\hat{Q}$ has the same invariant distribution $\pi$ as $Q$, and  that  $\hat{Q}$ is regular if $Q$ is regular, see Theorem 4.70 in \cite{Chen2004}.
Furthermore, $\mathcal{A}$ and  $\hat{\mathcal{A}}$ lead to a self-adjoint operator $\bar{\mathcal{A}}=({\mathcal{A}+\hat{\mathcal{A}}})/{2}$.
Similarly, we obtain a reversible $Q$-matrix $\bar{Q}=\big(\bar{Q}(i,j)\big)_{i,j\in E}$ with respective to the same  stationary distribution $\pi$ as follows
\begin{equation}\label{symmetrizing procedure-2}
	   \bar{Q}(i,j)=\frac{Q(i,j)+\hat{Q}(i,j)}{2}, \ \ \forall i,j\in E.
\end{equation}
According to Corollary 9.3 and Theorem 9.12 in 	\cite{Chen2004}, we know that
\begin{equation}\label{bar{q}}
\lambda(Q)=\lambda(\bar{Q}).
\end{equation}

Our main result is the following theorem, which yields a Hoeffding type inequality for CTMCs with the $\mathcal{L}^{2}(\pi)$-spectral gap $\lambda(Q)$.

\begin{theorem}\label{main result}
Assume that the chain $X_t$ admits an $\mathcal{L}^{2}(\pi)$-spectral gap $\lambda(Q)$, and that the $Q$-matrix $\bar{Q}$ is regular.
Then for any bounded function $g:E\to[a,b]$, and $t, \varepsilon >0$, we have
\[
		\mathbb{P}_{\pi}\left(\frac{1}{t} \int_{0}^{t} g\left(X_{s}\right)\mathrm{d}s-\pi (g) \geq \varepsilon \right) \leq \exp\left\{-\frac{\lambda(Q)t\varepsilon^2}{(b-a)^2} \right\}.
\]
\end{theorem}

\begin{remark}
Although both $Q$ and $\hat{Q}$ are regular,  $\bar{Q}$ may be not regular, see \cite{chen2000equivalence} for example.
Hence the regularity assumption of $\bar{Q}$ can not be removed.
\end{remark}

By H{\"o}lder's inequality, we have the following result directly.

\begin{corollary}\label{absolutely-continuous}
Assume that the assumptions of Theorem \ref{main result} hold.
If the initial measure $\nu$ is absolutely continuous with respect to the invariant distribution $\pi$,
then for any bounded function $g:E\to[a,b]$, and $t, \varepsilon >0$, we have
\[
		\mathbb{P}_{\nu}\left(\frac{1}{t} \int_{0}^{t} g\left(X_{s}\right)\mathrm{d}s-\pi (g) \geq \varepsilon \right) \leq \left\|\frac{\mathrm{d}\nu}{\mathrm{d}\pi}\right\|_{\pi,p}\exp\left\{-\frac{\lambda(Q)t\varepsilon^2}{q(b-a)^2} \right\},
\]
where
\[\left\|\frac{\mathrm{d} \nu}{\mathrm{d} \pi}\right\|_{\pi, p}:= \begin{cases}{\left[\pi\left(\left|\frac{\mathrm{d} \nu}{\mathrm{d}  \pi}\right|^p\right)\right]^{1 / p}}, & \text { if } 1<p<\infty, \\ \operatorname{ess} \sup \left|\frac{\mathrm{d} \nu}{\mathrm{d} \pi}\right|, & \text { if } p=\infty,\end{cases}\]
and
\[\frac{1}{p}+\frac{1}{q}=1.\]
\end{corollary}

\begin{remark}\label{remark-1}
L{\'e}zaud \cite{1998Chernoff} was the first to obtain a chernoff-type bound for a finite state space CTMC via the $\mathcal{L}^{2}(\pi)$-spectral gap $\lambda(Q)$.
If the function $g$ satisfies $\pi(g)=0$, $\|g\|_{\infty}=\|g\|_{\pi,2}^2=1$, then Theorem 3.4 in \cite{1998Chernoff} gives
\begin{equation*}\label{lezaud}
		\mathbb{P}_{\pi}\left(\frac{1}{t} \int_{0}^{t} g\left(X_{s}\right)\mathrm{d}s \geq \varepsilon \right) \leq \exp\left\{-\frac{\lambda(Q)t\varepsilon^2}{12} \right\}.
\end{equation*}
Under the same assumptions, our result derives a better bound as follows
\begin{equation*}
		\mathbb{P}_{\pi}\left(\frac{1}{t} \int_{0}^{t} g\left(X_{s}\right)\mathrm{d}s \geq \varepsilon \right) \leq \exp\left\{-\frac{\lambda(Q)t\varepsilon^2}{4} \right\}.
\end{equation*}
\end{remark}

Note that under the same assumptions as Theorem \ref{main result} or Corollary \ref{absolutely-continuous}, Hoeffding's inequalities are also established by replacing $\lambda(Q)$ by one of its lower bounds.
In the  reversible case, if there exists a constant $\beta>0$ and a fixed state $j\in{E}$ such that
\begin{equation}\label{kappa}
\mathbb{E}_{i}\left[e^{\beta\tau_j}\right]<\infty, \ \ \forall{i}\in{E},
\end{equation}
where $\tau_j:=\inf\{t\geq0: X_t=j\}$ is the first hitting time to the state $j$, then
$$\lambda(Q)\geq{\beta}.$$
Furthermore, (\ref{kappa}) holds if and only if there exists a function $V\geq1$ such that
\[QV(i)\leq-\beta{V}(i),\ \ i\neq{j},\]
see Theorem 4.1 in \cite{chen2000equivalence}.

In addition,  the $\mathcal{L}^{2}(\pi)$-spectral gap  $\lambda(Q)$ can be characterized by its Dirichlet form $\left(D,\mathcal{D}(D)\right)$,
which will be used in the proof of Theorem \ref{main result},
where $D$ is an operator on $\mathcal{L}^{2}(\pi)$, defined by
\[
	D(f):=\lim_{t\to 0}\frac{\langle f-P_tf,f \rangle_\pi}{t},\ \ \forall f\in \mathcal{L}^{2}(\pi),
\]
and $\mathcal{D}(D)$ is the domain of $D$, given by
\[
\mathcal{D}(D)=\left\{f\in \mathcal{L}^{2}(\pi):D(f)<\infty\right\}.
\]
The spectral gap of $D$ is defined by
\[
		\lambda(D):=\inf \left\{D(f): f\in \mathcal{D}(D),\ \|f\|_{\pi,2}=1,\ \pi(f)=0 \right\}.
\]
From Theorem 9.1 in \cite{Chen2004}, we know that
\begin{equation}\label{D-Q}
\lambda(Q)=\lambda(D).
\end{equation}

There is another Dirichlet form $\left(D^*,\mathcal{D}(D^*)\right)$, where $D^{*}$ is an operator  on $\mathcal{L}^{2}(\pi)$, defined by
\[D^{*}(f):=\frac{1}{2}\sum_{i,j\in E}\pi(i){Q}(i,j)\left(f(j)-f(i)\right)^2,\ \ \forall f\in \mathcal{L}^{2}(\pi),\]
and  $\mathcal{D}(D^{*})$ is the domain of $D^{*}$, given by
\[\mathcal{D}(D^{*})=\left\{f \in \mathcal{L}^{2}(\pi):D^{*}(f)<\infty\right\}.\]
If $Q$-matrix $\bar{Q}$ is regular, it follows from \cite{chen2000equivalence} that $D=D^{*}$ and $\mathcal{D}(D)=\mathcal{D}(D^{*})$, i.e.,
\begin{equation}\label{dense}
\mathcal{K} \text { is dense in } \mathcal{D}\left(D^*\right) \text { in the }\|\cdot\|_{D^*-\operatorname{norm}}\left(\|f\|_{D^*}^2:=\|f\|^2+D^*(f)\right),
\end{equation}
where $\mathcal{K}$ is the set of all functions on $E$ with finite support.

\begin{example}[Irreversible Case]
Consider a CTMC on a finite state space $E=\left\{0,1,2\right\}$ with the $Q$-matrix as follows
\[
		Q=\left(\begin{array}{ccc}
			-2&1&1\\
			1&-3&2\\
			1&0&-1\\
		\end{array} \right).
\]
\end{example}
By calculations, we have
\[\pi(0)=\frac{1}{3},\ \ \pi(1)=\frac{1}{9},\ \ \pi(2)=\frac{5}{9}.\]
According to (\ref{symmetrizing procedure-1}) and (\ref{symmetrizing procedure-2}), we can obtain
\[
		\hat{Q}=\left(\begin{array}{ccc}
			-2&1/3&5/3\\
			3&-3&0\\
			3/5&2/5&-1\\
		\end{array} \right),\ \
		\bar{Q}=\left(\begin{array}{ccc}
			-2&2/3&4/3\\
			2&-3&1\\
			4/5&1/5&-1\\
		\end{array} \right).
\]
It is easy to calculate that
\[\bar{\lambda}_1=\frac{15-\sqrt{15}}{5},\ \ \bar{\lambda}_2=\frac{15+\sqrt{15}}{5}.\]
Then, combining with (\ref{eigenvalue}) and (\ref{bar{q}}), we have
$$\lambda(Q)=\frac{{15-\sqrt{15}}}{5}.$$

\begin{example}[Birth and Death Processes]
Consider a birth and death process with death rates $a_i$, $1\leq i\leq N$, and birth rates $b_i$, $0\leq i\leq N-1$, $N\leq\infty$.
\end{example}
In the case of $a_i=\alpha$ and $b_i=\beta$,
it follows from \cite{granovsky1997decay} that
\[\lambda(Q)=\alpha+\beta-{2\sqrt{\alpha\beta}}\cos(\frac{\pi}{N+1}).\]
In particular, as $N\rightarrow\infty$ and $\alpha>\beta$, we have
\[\lambda(Q)=\alpha+\beta-\lim_{N\rightarrow\infty}{2\sqrt{\alpha\beta}}\cos(\frac{\pi}{N+1})=(\sqrt{\alpha}-\sqrt{\beta})^2.\]

For general birth and death processes, we can give an estimate of the lower bound of the spectral gap.
Define
		$$
		\mu_1=1,\ \ \mu_k=\frac{b_1\cdots b_{k-1}}{a_2\cdots a_k},\ \ 2\leq k\leq N\leq \infty.
		$$
If $\sum_{i=0}^{N}\mu_i<\infty$ and $\delta:=\sup_{n\in E}\sum_{j=0}^{n}\mu_j\sum_{k=n}^{N}\frac{1}{\mu_kb_k}<\infty$, then it follows from Corollary 6.4 in \cite{chen2010speed} that
\[
	\lambda(Q)\geq \frac{1}{4\delta}.
\]

\section{Proof of Theorem 2.2}\label{sec 3}

To prove Theorem \ref{main result}, we initially introduce Hoeffding's inequality for DTMCs based on the $\mathcal{L}^{2}(\pi)$-spectral gap.
Let $\left\{X_{k}, k\geq 0\right\}$ be an irreducible and positive recurrent DTMC on the countable state space $E$.
Let $P=(P(i,j))_{i,j\in E}$ and $\pi$ be the one-step transition matrix and the invariant probability distribution of the chain $X_k$.
It is known that each transition matrix $P$ of DTMC can be viewed as a linear operator acting on $\mathcal{L}^{2}(\pi)$, which acts to the right on functions, i.e.,
\[Pf(i)=\sum_{j\in E}P(i,j)f(j), \ \ \forall i \in E, \forall f \in \mathcal{L}^{2}(\pi).\]
We now introduce the $\mathcal{L}^{2}(\pi)$-spectral gap of $P$, denoted by $1-\lambda(P)$.

\begin{definition}\label{spectral-gap-DTMC}
The chain  $X_{k}$ is said to admit an  $\mathcal{L}^{2}(\pi)$-spectral gap $1-\lambda(P)$ if
\begin{equation}\label{right-spectral-gap}
		\lambda(P):=\sup \left\{\langle{P}h,h\rangle_{\pi}:\|h\|_{\pi,2}=1,\ \pi(h)=0\right\}<1.
\end{equation}
\end{definition}

Note that when the state space $E$ is finite and the transition matrix $P$ is reversible, i.e., $P$ satisfies
\[\pi(i)P(i,j)=\pi(j)P(j,i),  \ \ \forall i,j\in E,\]
then $\lambda(P)$ coincides with the second largest eigenvalue of $P$ (the largest eigenvalue of $P$ is identically equal to 1).
Similar to CTMCs, it is not difficult to verify that
\begin{equation}\label{bar{p}}
\lambda(P)=\lambda(\bar{P}),
\end{equation}
where
\begin{equation*}
\bar{P}(i,j)=\frac{P(i,j)+\hat{P}(i,j)}{2}\ \  \text{and}\ \ \hat{P}(i,j)=\frac{\pi(j)P(j,i)}{\pi(i)}, \ \ \forall i,j\in E.
\end{equation*}

According to Theorem 3.3 in \cite{fan2021hoeffding} and (\ref{bar{p}}),
we introduce the following Hoeffding inequality regarding DTMCs with the $\mathcal{L}^{2}(\pi)$-spectral gap $1-\lambda(P)$.

\begin{theorem}\label{main result dtmc}
Assume that the chain $X_k$ admits  an  $\mathcal{L}^{2}(\pi)$-spectral gap $1-\lambda(P)$.
Then for any bounded function $g:E\to[a,b]$, and $n, \varepsilon >0$, we have
\[
		\mathbb{P}_{\pi}\left(\frac{1}{n} \sum_{k=0}^{n-1} g\left(X_{k}\right)-\pi (g) \geq \varepsilon \right) \leq \exp\left\{-\frac{1-\max\{\lambda(P),0\}}{1+\max\{\lambda(P),0\}}\cdot\frac{2n\varepsilon^2}{(b-a)^2} \right\}.
\]
\end{theorem}

Next, we will derive the proof of Theorem \ref{main result} into three steps, which are presented in Subsections 3.1-3.3.
In Subsection \ref{ssub-1}, we first introduce a Hoeffding inequality for finite state CTMCs, employing Theorem \ref{main result dtmc} and the skeleton chain technique.
In Subsection \ref{ssub-2}, we introduce the collapsed chain and present some basic properties.
Finally in  Subsection \ref{ssub-3},  we extend the result established for finite state CTMCs to infinitely state by employing the results established in  Subsections \ref{ssub-1} and \ref{ssub-2}.

\subsection{Finite state CTMCs}\label{ssub-1}

In this subsection, we present a Hoeffding inequality for a finite state CTMC.
To show this, we introduce the definition of the skeleton chain according to Chapter 5 of \cite{Anderson1991}.
Given a number $\delta>0$, the DTMC $\left\{X_{n\delta}, n\geq 1\right\}$ having the one-step transition matrix $P^\delta=(P^{\delta}(i,j))_{i,j\in{E}}$
(and therefore the $n$-step  transition matrix $P^{n\delta}$) is called the $\delta$-skeleton chain of the CTMC $\left\{X_{t}, t\geq 0\right\}$.
It is not difficult to verify that $X_{n\delta}$ and $X_t$ have the same invariant probability distribution $\pi$.
In addition, if $X_t$ is reversible, then $X_{n\delta}$ is also reversible.
Furthermore, if $X_t$ has a finite state space, the transition matrix $P^\delta$ is given by
\begin{equation}\label{skeleton-chain}
P^\delta=\exp\left\{{\delta}Q\right\}=\sum_{n=0}^{\infty}\frac{\delta^n{Q^n}}{n!}.
\end{equation}

\begin{theorem}\label{theroem-1}
Assume that the chain $X_{t}$ has a finite state space $E$.
Then for any bounded function $g:E\to [a,b]$, and $t, \varepsilon >0$, we have
\[
\mathbb{P}_{\pi}\left(\frac{1}{t} \int_{0}^{t} g\left(X_{s}\right)\mathrm{d}s-\pi(g) \geq \varepsilon \right) \leq\exp\left\{-\frac{\lambda(Q)t\varepsilon^2}{(b-a)^2}\right\}.
\]
\end{theorem}
\begin{proof}
For any $\omega\in\Omega$, $X_t(\omega)$ is a right-continuous with left limits sample function.
Furthermore, since $g$ is bounded,
the integral $\frac{1}{t} \int_{0}^{t} g\left(X_{s}\right)\mathrm{d}s$ can be expressed as the limit of Riemann sums, such as
\[
		\frac{1}{t} \int_{0}^{t} g\left(X_{s}(\omega)\right)\mathrm{d}s=\lim _{k \rightarrow \infty} \frac{1}{k} \sum_{i=0}^{k-1} g\left(X_{i t / k}(\omega)\right),
\]
from which, we know that
\begin{equation}\label{theorem-3.2-1}
\frac{1}{k} \sum_{i=0}^{k-1} g\left(X_{i t / k}\right)\xrightarrow{a.s.}\frac{1}{t} \int_{0}^{t} g\left(X_{s}\right)\mathrm{d}s,
\end{equation}
where $X_{it/k}$ is the $\frac{t}{k}$-skeleton chain of $X_t$.
It follows from (\ref{skeleton-chain}) that
\begin{equation}\label{theorem-3.2-1-1}
P^{\frac{t}{k}}=\sum_{n=0}^{\infty}\frac{t^n}{n!\cdot k^n}Q^n=I+\frac{t}{k}Q+o(\frac{1}{k}),
\end{equation}
where $I$ is the identity matrix and $o(\frac{1}{k})$ is the infinitesimal of higher order with respect to $\frac{1}{k}$.

According to (\ref{right-spectral-gap}) and (\ref{theorem-3.2-1-1}), we have
\begin{eqnarray}\label{P-Q-1}
 \nonumber \lambda(P^{\frac{t}{k}}) &=&\sup \left\{\langle P^{\frac{t}{k}}f,f \rangle_{\pi} :\|f\|_{\pi,2}=1,\pi(f)=0 \right\}\\
  \nonumber &=& \sup \left\{ \bigg<  \left(I+\frac{t}{k}{Q}+o(\frac{1}{k}) \right)f,f \bigg>_{\pi} :\|f\|_{\pi,2}=1,\pi(f)=0 \right\}\\
  \nonumber &=& \sup \left\{1+\Big< \frac{t}{k}{Q}f,f \Big>_{\pi}+o(\frac{1}{k}) :\|f\|_{\pi,2}=1,\pi(f)=0 \right\}\\
\nonumber   &=& 1+\sup \left\{\Big< \frac{t}{k}{Q}f,f \Big>_{\pi}:\|f\|_{\pi,2}=1,\pi(f)=0 \right\}+o(\frac{1}{k})\\
  &=& 1-\frac{t}{k} \inf \left\{-\langle{ Q}f,f \rangle_{\pi}:\|f\|_{\pi,2}=1,\pi(f)=0 \right\}+o(\frac{1}{k}).
\end{eqnarray}
Since the state space $E$ is finite, it is evident from (\ref{eigenvalue}) and (\ref{bar{q}}) that $\lambda(Q)>0$.
In addition, it follows that for any $f\in\mathcal{L}^{2}(\pi)$, we have $f\in\mathcal{D}(\mathcal{A})$.
Thus, when the state space $E$ is finite, the infinitesimal generator $\mathcal{A}$ is equal to the $Q$-matrix $Q$.
According to Definition \ref{spectral-gap}, we can obtain that
\begin{equation}\label{P-Q-2}
\lambda(Q)=\inf \left\{-\langle	{Q} f,f \rangle_{\pi}:  \|f\|_{\pi,2}=1,\ \pi(f)=0 \right\}>0.
\end{equation}
Hence, from (\ref{P-Q-1}) and (\ref{P-Q-2}), we have
\begin{equation}\label{P-Q-3}
\lambda(P^{\frac{t}{k}})= 1-\frac{t}{k}\lambda(Q)+o(\frac{1}{k}).
\end{equation}

Since $\lambda(Q)>0$, then for any fixed $t>0$, there exists a positive constant $K$ such that
$$-1<-\frac{t}{k}\lambda(Q)+o(\frac{1}{k})<0,\ \ \forall k>K,$$
which implies that
\begin{equation}\label{12345}
0<\lambda(P^{\frac{t}{k}})<1,\ \ \forall k>K.
\end{equation}
Furthermore, according to (\ref{P-Q-3}), we have
\begin{equation}\label{54321}
1-\lambda(P^{\frac{t}{k}})=\frac{t}{k}\lambda(Q)+o(\frac{1}{k}).
\end{equation}

Therefore, as $k>K$ and for any $\varepsilon >0$, Theorem \ref{main result dtmc} and (\ref{12345})--(\ref{54321}) give
\begin{eqnarray}\label{3,5}
 \nonumber \mathbb{P}_{\pi}\left(\frac{1}{k}\sum_{i=0}^{k-1} g\left(X_{it/k}\right)-\pi (g) \geq \varepsilon \right) & \leq &  \exp\left\{-2\cdot\frac{1-\lambda(P^{\frac{t}{k}})}{1+\lambda(P^{\frac{t}{k}})}\cdot \frac{k\varepsilon^2}{(b-a)^2} \right\} \\
\nonumber &=&\exp\left\{-2\cdot\frac{\frac{t}{k}{\lambda}(Q)+o(\frac{1}{k})}{2-\frac{t}{k}{\lambda}(Q)+o(\frac{1}{k})}\cdot \frac{k\varepsilon^2}{(b-a)^2} \right\} \\
  &=& \exp\left\{-2\cdot\frac{t{\lambda}(Q)+ k \cdot o(\frac{1}{k})}{2-\frac{t}{k}{\lambda}(Q)+o(\frac{1}{k})}\cdot \frac{\varepsilon^2}{(b-a)^2} \right\}.
\end{eqnarray}
Applying (\ref{theorem-3.2-1}) and Fatou's lemma, and letting  $k\to \infty$ to inequality (\ref{3,5}) gives the following result
\begin{eqnarray*}
  \mathbb{P}_{\pi}\left(\frac{1}{t} \int_{0}^{t} g\left(X_{s}\right)\mathrm{d}s-\pi (g) \geq \varepsilon \right) &=& \mathbb{P}_{\pi}\left(\lim_{k\to \infty}\frac{1}{k}\sum_{i=0}^{k-1} g\left(X_{it/k}\right)-\pi (g) \geq \varepsilon \right)\\
   &\leq& \mathop{\liminf}\limits_{k\to \infty} \mathbb{P}_{\pi}\left(\frac{1}{k}\sum_{i=0}^{k-1} g\left(X_{it/k}\right)-\pi (g) \geq \varepsilon \right)\\
   &\leq& \mathop{\liminf}\limits_{k\to \infty} \exp\left\{-2\cdot\frac{t{\lambda}(Q)- k\cdot o(\frac{1}{k})}{2-\frac{t}{k}{\lambda}(Q)+o(\frac{1}{k})}\cdot \frac{\varepsilon^2}{(b-a)^2} \right\}\\
   &=&  \exp\left\{\lim_{k\to \infty} -2\cdot\frac{t{\lambda}(Q)- k \cdot o(\frac{1}{k})}{2-\frac{t}{k}{\lambda}(Q)+o(\frac{1}{k})}\cdot \frac{\varepsilon^2}{(b-a)^2} \right\}\\
  &=&\exp\left\{-\frac{{\lambda}(Q)t\varepsilon^2}{(b-a)^2} \right\}.
\end{eqnarray*}
Thus, we obtain the assertion.
\end{proof}

Using the same proof of Theorem \ref{theroem-1}, we can obtain the Hoeffding inequality for a infinite state CTMC with a uniformly bounded $Q$-matrix (i.e., $\sup_{i\in E}-Q(i,i)<\infty$).

\begin{corollary}\label{corollary-1}
Assume that the chain $X_t$ admits an $\mathcal{L}^{2}(\pi)$-spectral gap $\lambda(Q)$, and that the $Q$-matrix $Q$ is uniformly bounded.
Then for any bounded function $g:E\to [a,b]$, and $t, \varepsilon >0$, we have
\[
\mathbb{P}_{\pi}\left(\frac{1}{t} \int_{0}^{t} g\left(X_{s}\right)\mathrm{d}s-\pi(g) \geq \varepsilon \right) \leq\exp\left\{-\frac{\lambda(Q)t\varepsilon^2}{(b-a)^2}\right\}.
\]
\end{corollary}

However, since equality (\ref{skeleton-chain}) fails to hold for a CTMC with an unbounded $Q$-matrix, the skeleton chain technique is not applicable to more general CTMCs, which urges us to develop the augmented truncation approximation technique in order to extend the results from finite state CTMCs to infinite state CTMCs.

\subsection{The collapsed chain}\label{ssub-2}

We first introduce the technique of the collapsed chain, one special augmented truncation approximation.
Let $\{{ }_{(n)}E,n=1,2,\ldots\}$ be a sequence of subsets in the state space $E$, satisfying ${ }_{(n)}E\subset{ }_{(n+1)}E$ and $\lim_{n\rightarrow\infty}{ }_{(n)}E=E$.
Then, let ${ }_{(n)}E^C$ be the complement of ${ }_{(n)}E$ and consider $_{(n)}e={ }_{(n)}E^C$ as a new single-point state.
For an irreducible and regular $Q$-matrix $Q=(Q(i,j))_{i,j\in{E}}$,
consider the augmented truncation $Q$-matrix ${ }_{(n+1)}\widetilde{Q}=({ }_{(n+1)}\widetilde{Q}(i,j))_{i,j\in{{ }_{(n+1)}\widetilde{E}}}$ on the state space ${ }_{(n+1)}\widetilde{E}={ }_{(n)}E\cup\{{_{(n)}e}\}$,
given by
\begin{equation*}\label{nQ-1}
{ }_{(n+1)}\widetilde{Q}(i,j)=Q(i,j),\ \ i,j\in{{ }_{(n)}E},
\end{equation*}
\begin{equation*}\label{nQ-2}
{ }_{(n+1)}\widetilde{Q}(i,{_{(n)}e})=\sum_{k\in{{ }_{(n)}E^C}}Q(i,k),\ \ i\in{{ }_{(n)}E},
\end{equation*}
\begin{equation*}\label{nQ-3}
{ }_{(n+1)}\widetilde{Q}({_{(n)}e},i)=\frac{\sum_{k\in{{ }_{(n)}E^C}}\pi(k)Q(k,i)}{\sum_{k\in{{ }_{(n)}E^C}}\pi(k)},\ \ i\in{{ }_{(n)}E},
\end{equation*}
and
\begin{equation*}\label{nQ-4}
{ }_{(n+1)}\widetilde{Q}({_{(n)}e},{_{(n)}e})=-\sum_{k\in{{ }_{(n)}E}}{ }_{(n+1)}\widetilde{Q}({_{(n)}e},k).
\end{equation*}

Denote by $\big\{{ }_{(n+1)}\widetilde{X}_{t}, t\geq 0\big\}$ the CTMC with $Q$-matrix ${ }_{(n+1)} \widetilde{Q}$, whose irreducibility is inherent from the irreducibility of the original chain. The chain $_{(n+1)}\widetilde{X}_t$ is called the collapsed chain since it can be  interpreted as ``the original chains with states collapsed to a single state ${_{(n)}e}$''.
This collapsed chain was essentially introduced in Chapter 9 of \cite{Chen2004}  for continuous-time reversible Markov chains.
Here we do not need the reversibility assumption. The discrete-time version of the collapsed chain was introduced in Chapter 2 of \cite{Aldous02}.
Let ${ }_{(n+1)}\widetilde{\pi}(i)=\pi(i), i\in{{ }_{(n)}E}, { }_{(n+1)}\tilde{\pi}({_{(n)}e})=\sum_{k\in{{ }_{(n)}E^C}}\pi(k)$.
It is not difficult to verify that ${ }_{(n+1)}\widetilde{\pi}$ is the unique stationary distribution of the chain ${ }_{(n+1)}\widetilde{X}_{t}$.

Let ${_{(n+1)}\widetilde{g}}$ be an $(n+1)$-dimensional function on ${{ }_{(n+1)}\widetilde{E}}$, such that ${ }_{(n+1)}\widetilde{g}(i)=g(i)$, $i\in{_{(n)}E}$ and ${ }_{(n+1)}\widetilde{g}({_{(n)}e})=0$.
For the chain ${{ }_{(n)}\widetilde{X}_{t}}$ and the function $_{(n)}\widetilde{g}$,  we have the following lemma.

\begin{lemma}\label{3-1}
For any bounded function $g$ on $E$, we have
		$$
		\mathbb{P}_i\left(\omega\in \varOmega:\lim_{n \rightarrow \infty}{_{(n)}\widetilde{g}}\big({ }_{(n)}\widetilde{X}_{s}(\omega) \big)=g\left(X_{s}(\omega)\right) \right)=1,
		$$
and
		$$
		\lim_{n \rightarrow \infty} { }_{(n)}\widetilde{\pi}\left({ }_{(n)}\widetilde{g} \right)=\pi(g).
		$$	
	\end{lemma}	
\begin{proof}
		Since $Q$-matrix $Q$ is regular, we know that the chain $X_t$ is non-explosive, i.e., for any fixes $s>0$,
	$$
		\mathbb{P}_{i}\left(\omega \in \Omega: J(s,\omega)<\infty\right)=1,
		$$
where $J(s)$ denotes the number of jumps of $X_t$ in the time interval $(0, s)$.
Furthermore, for any sample path $\omega$, we define $M_s(\omega)$ as a subset of $E$ that includes all states visited by $X_t(\omega)$ during the time interval $(0, s)$.
Hence, we can find an $N$ such that $M_s(\omega)\subset{{ }_{(n)}E}$ for all $n>N$.
Due to the particular structure of ${ }_{(n)}\widetilde{Q}$, we have ${ }_{(n)}\widetilde{X}_{t}(\omega)=X_{t}(\omega), 0 \leq t \leq s$, $n>N$.
That is, for any fixed $s>0$,
		$$
		\mathbb{P}_{i}\left(\omega \in \Omega: \lim _{n \rightarrow \infty}{ }_{(n)}\widetilde{X}_{s}(\omega)=X_{s}(\omega)\right)=1 .
		$$
		Furthermore, for any bounded function $g$ on $E$, we have
		$$
		\mathbb{P}_{i}\left(\omega \in \Omega: \lim _{n \rightarrow \infty}{_{(n)}\widetilde{g}}\big({}_{(n)}\widetilde{X}_{s}\left(\omega\right)\big)=g\left(X_{s}\left(\omega\right)\right)\right)=1.
		$$
		\noindent
		
On the other hand, by applying the dominated convergence theorem, we can conclude that
		\begin{equation*}
			\begin{split}
				\lim_{n \rightarrow \infty} { }_{(n)}\widetilde{\pi}\left({ }_{(n)}\widetilde{g} \right)&=\lim_{n \rightarrow \infty}\sum_{i\in{{ }_{(n-1)}E}}\pi(i) g(i)\\
				&=\lim_{n \rightarrow \infty}\sum_{i\in{E}}\pi(i){g}(i)\mathbf{1}_{\{i\in {{{ }_{(n-1)}\widetilde{E}}}\}} \\
				&=\pi(g),
			\end{split}
		\end{equation*}
where $\mathbf{1}_{\{\cdot\}}$ is the indicator function. Thus, we complete the proof of this lemma.
\end{proof}

\begin{lemma}\label{3-2}
Assume that the original chain $X_t$ admits an $\mathcal{L}^{2}(\pi)$-spectral gap $\lambda(Q)$, and that the $Q$-matrix $\bar{Q}$ is regular.
Then we have
\[
	\lim_{n\rightarrow\infty}	{\lambda}({ }_{(n)}\widetilde{Q}) = {\lambda}(Q).
\]
\end{lemma}
\begin{proof}
Since $\bar{Q}$ is regular, according to (\ref{D-Q}) and (\ref{dense}), we have
\begin{eqnarray}
\nonumber  \lambda(D) &=& \inf \left\{D^*(f): \ \|f\|_{\pi,2}=1,\ \pi(f)=0 \right\}\\
\label{nd-1}&=& \inf \left\{D^*(f): f\in\mathcal{K},\ \|f\|_{\pi,2}=1,\ \pi(f)=0 \right\},
\end{eqnarray}
where $\mathcal{K}$ is the set of all functions on $E$ with finite support.
Without loss of generality, support that $f(i)=d$, $i\in{_{(n-1)}E^C}$, for some $n>1$, where $d$ is a constant.
The rest of the proof is modified from the proof of Theorem 9.20 in \cite{Chen2004}.
It is simple to show that $\pi(f)=0$, $\|f\|_{\pi,2}=1$ if and only if ${ }_{(n)}\widetilde{\pi}({_{(n)}{f}})=0$, $\|{_{(n)}{f}}\|_{{ }_{(n)}\widetilde{\pi},2}=1$.
Now, let
\[{ }_{(n)}{D}^{\ast}({f})=\frac{1}{2}\sum_{i,j\in{_{(n)}E}}{ }_{(n)}\widetilde{\pi}(i){ }_{(n)}{\widetilde{Q}}(i,j)(f(i)-f(j))^2, \ \  \forall f\in{\mathcal{L}^2({ }_{(n)}\widetilde{\pi})}.\]
Furthermore, we have
\begin{eqnarray*}\label{nd-2}
 \nonumber {D}^{\ast}(f) &=& \frac{1}{2}\sum_{i,j\in E} \pi(i){Q}(i,j)\left(f(j)-f(i)\right)^2 \\
\nonumber   &=& \frac{1}{2}\sum_{i\in { }_{(n-1)}E}\pi(i)\sum_{j\in { }_{(n-1)} E}{Q}(i,j)\left(f(j)-f(i)\right)^2+\frac{1}{2}\sum_{i\in { }_{(n-1)}E}\pi(i)\sum_{j\in { }_{(n-1)} E^{c}}{Q}(i,j)\left(d-f(i)\right)^2\\
\nonumber   &&+\frac{1}{2}\sum_{i\in { }_{(n-1)}E^C}\pi(i)\sum_{j\in { }_{(n-1)} E}{Q}(i,j)\left(d-f(j)\right)^2\\
   \nonumber&=& \frac{1}{2}\sum_{i\in { }_{(n-1)}E}{ }_{(n)}\widetilde{\pi}(i)\sum_{j\in { }_{(n-1)}E}{ }_{(n)}{\widetilde{Q}}(i,j)\left(f(j)-f(i)\right)^2+\frac{1}{2}\sum_{i\in { }_{(n-1)}E}{ }_{(n)}\widetilde{\pi}(i){ }_{(n)}{\widetilde{Q}}(i,n)\left(d-f(i)\right)^2\\
      \nonumber&&+\frac{1}{2} \sum_{i\in { }_{(n-1)}E}{ }_{(n)}\tilde{\pi}(n){ }_{(n)}{\widetilde{Q}}(n,i)\left(d-f(i)\right)^2\\
  \nonumber &=&\frac{1}{2} \sum_{i\in { }_{(n)}E}{ }_{(n)}\widetilde{\pi}(i)\sum_{j\in { }_{(n)}E}{ }_{(n)}{\widetilde{Q}}(i,j)\left({ }_{(n)}{f}(j)-{ }_{(n)}{f}(i)\right)^2\\
   &=&{ }_{(n)}{D}^{\ast}({_{(n)}{f}}).
\end{eqnarray*}
From $(\ref{nd-1})$ and the above equation, we know that
\begin{eqnarray}\label{nd-3}
\nonumber  {\lambda}({Q}) &=& \inf \left\{{D}^{\ast}(f):\pi(f)=0,\|f\|_{\pi,2}=1, f(i)=d\ \text{for}\ i\in{_{(n-1)}E^C}\ \text{and some}\ n>1 \right\} \\	
\nonumber   &=& \lim_{n \rightarrow \infty} \inf \left\{{D}^{\ast}(f):\pi(f)=0,\|f\|_{\pi,2}=1, f(i)=d\ \text{for}\ i\in{_{(n-1)}E^C}\ \text{and some}\ n>1 \right\}\\
 \nonumber &=& \lim_{n \rightarrow \infty} \inf \left\{{ }_{(n)}{D}^{\ast}({_{(n)}{f}}):{ }_{(n)}\widetilde{\pi}({_{(n)}f})=0,\|{_{(n)}f}\|_{{_{(n)}\widetilde{\pi}},2}=1 \right\}\\
   &=&\lim_{n \rightarrow \infty}{\lambda}({ }_{(n)}\widetilde{Q}).
\end{eqnarray}
Thus, combining with (\ref{D-Q}) and (\ref{nd-3}), we obtain the assertion immediately.
	\end{proof}	

\subsection{Infinite state CTMCs}\label{ssub-3}

We now finish the proof of Theorem \ref{main result} by extending Hoeffding's inequality from finite state CTMCs to infinite state CTMCs.

\begin{proof}[Proof of Theorem \ref{main result}]
According to Theorem \ref{theroem-1}, for any bounded function $g:E\to [a,b]$ and $n\geq 1$, we have
	\begin{equation}\label{main-reslute-1}
		\mathbb{P}_{_{(n)}\widetilde{\pi}}\left(\frac{1}{t} \int_{0}^{t} {_{(n)}\widetilde{g}}\left({ }_{(n)}\widetilde{X}_{s}\right)\mathrm{d}s-{ }_{(n)}\pi ({_{(n)}\widetilde{g}}) \geq \varepsilon \right) \leq \exp\left\{-\frac{{\lambda}({ }_{(n)}\widetilde{Q})t}{(b-a)^2}\varepsilon^2 \right\}.
	\end{equation}
Then, applying Lemmas \ref{3-1}-\ref{3-2} and Fatou's Lemma, and letting $n\to \infty$ to inequality (\ref{main-reslute-1}), we have
\begin{eqnarray*}
			&&\mathbb{P}_{\pi}\left(\frac{1}{t} \int_{0}^{t} g\left(X_{s}\right)\mathrm{d}s-\pi (g) \geq \varepsilon \right)\\
             &=&\sum_{i\in{E}}\pi(i)\mathbb{P}_{i}\left(\frac{1}{t} \int_{0}^{t} g\left(X_{s}\right)\mathrm{d}s-\pi (g) \geq \varepsilon \right)\\
             &=&\sum_{i\in{E}}\lim_{n\rightarrow\infty}\mathbf{1}_{\{i\in{_{(n)}\widetilde{E}}\}}{_{(n)}\widetilde{\pi}(i)}\mathbb{P}_{i}\left(\frac{1}{t} \int_{0}^{t}\lim_{n \rightarrow \infty} {_{(n)}\widetilde{g}}\left({ }_{(n)}X_{s}\right)\mathrm{d}s-\lim_{n \rightarrow \infty}{ }_{(n)}\pi ({_{(n)}\widetilde{g}}) \geq \varepsilon \right)\\
             &\leq&\liminf_{n\rightarrow\infty}\sum_{i\in{_{(n)}\widetilde{E}}}{_{(n)}\widetilde{\pi}(i)}  \mathbb{P}_{i}\left(\frac{1}{t} \int_{0}^{t} {_{(n)}\widetilde{g}}\left({ }_{(n)}X_{s}\right)\mathrm{d}s-{ }_{(n)}\pi ({_{(n)}\widetilde{g}}) \geq \varepsilon \right)\\
			&=&\liminf_{n\rightarrow\infty}\mathbb{P}_{_{(n)}\widetilde{\pi}}\left(\frac{1}{t} \int_{0}^{t} {_{(n)}\widetilde{g}}\left({ }_{(n)}X_{s}\right)\mathrm{d}s-{ }_{(n)}\pi ({_{(n)}\widetilde{g}}) \geq \varepsilon \right)\\
			&\leq& \mathop{\lim\inf}\limits_{n\to \infty} \exp\left\{-\frac{\lambda({ }_{(n)}\widetilde{Q})t\varepsilon^2}{(b-a)^2} \right\}\\
			&=&\exp\left\{\lim_{n \rightarrow \infty}-\frac{\lambda({ }_{(n)}\widetilde{Q})t\varepsilon^2}{(b-a)^2} \right\}\\
			&=&\exp\left\{-\frac{{\lambda}(Q)t\varepsilon^2}{(b-a)^2} \right\}.
\end{eqnarray*}
\end{proof}

\section{Extensions to Jump Processes}\label{sec 4}

In this section, let $E$ be a general state space and $\mathcal{B}$ represent the associated {$\sigma$}-algebra.
Let $\{X_t,t\geq0\}$ be an irreducible time-homogeneous continuous-time Markov process on $E$ with the transition function  $P_t=(P_t(x,A))_{x\in{E},A\in{\mathcal{B}}}$ and the stationary distribution $\pi$.
We assume that $P_t$ is continuous at the origin, i.e., $X_t$ is a jump process. Please refer  to \cite{Chen2004} for more details about jump processes.
Furthermore, we focus on the case that $X_t$ has a $q$-pair $(q(x),q(x,A))$, i.e., for any
$x\in{E}$, $A\in\mathcal{B}$, we have
\[
\lim_{t\rightarrow0}\frac{P_t(x,A)-\mathbf{1}_{\{x\in{A}\}}}{t}=q(x,A)-q(x)\mathbf{1}_{\{x\in{A}\}},\ \ x\in{E},A\in\mathcal{B}.
\]
Furthermore, we assume that this $q$-pair $(q(x),q(x,A))$ is  totally stable, conservative and regular.
The infinitesimal generator $\mathcal{A}$ generated by $P_t$ is defined by (\ref{infinitestimal}). We define
\[
		\lambda(q):=\inf \left\{-\langle	{\mathcal{A}} f,f \rangle_{\pi}: f\in \mathcal{D}({\mathcal{A}}),\ \|f\|_{\pi,2}=1,\ \pi(f)=0 \right\},
\]
where $\mathcal{D}(\mathcal{A})$ is the domain of $\mathcal{A}$. If $\lambda(q)>0$, $\lambda(q)$ is said to be the $\mathcal{L}^{2}(\pi)$-spectral gap of $X_t$.

By the technique of skeleton chains, we can obtain the following result similar to Corollary \ref{corollary-1}.
\begin{lemma}\label{lemma-1-ctmc}
Assume that the process $X_t$ exists an $\mathcal{L}^{2}(\pi)$-spectral gap $\lambda(q)$, and that the $q$-pair $(q(x),q(x,A))$ is uniformly bounded.
Then for any bounded continuous function $g:E\to [a,b]$, and $t, \varepsilon >0$, we have
\[
\mathbb{P}_{\pi}\left(\frac{1}{t} \int_{0}^{t} g\left(X_{s}\right)\mathrm{d}s-\pi(g) \geq \varepsilon \right) \leq\exp\left\{-\frac{\lambda(q)t\varepsilon^2}{(b-a)^2}\right\}.
\]
\end{lemma}

Similar to the case of countable state space, define
\[D^{*}(f):=\frac{1}{2}\int_{x,y\in{E}}\pi(dx)q(x,dy)\left(f(x)-f(y)\right)^2,\ \ \forall{f}\in\mathcal{L}^2(\pi),\]
and
\[\mathcal{D}(D^{*}):=\left\{f \in \mathcal{L}^{2}(\pi):D^{*}(f)<\infty\right\},\]
where $\mathcal{D}(D^{*})$ is the domain of $D^{*}$.
Let $\mathcal{K}$ be the set of all functions on $E$ with finite support.
According to Theorem 9.11 in \cite{Chen2004}, we know that if  $\mathcal{K}$  is dense in $D^*$, then
\begin{equation*}\label{1}
\lambda(q)=\inf \left\{D^*(f): f\in\mathcal{K},\ \|f\|_{\pi,2}=1,\ \pi(f)=0 \right\}.
\end{equation*}

Now, let $\{{ }_{(n)}E,n=1,2,\ldots\}$ be  an increasing sequence compact subsets of $E$ such that
\begin{equation*}\label{pi(nEc)}
\int_{{ }_{(n)}E^C}\pi(dx)>0,\ \ n\geq1.
\end{equation*}
Furthermore, let $_{(n)}e={ }_{(n)}E^C$ as a new single-point state.
We consider the following collapsed $q$-pair $({_{(n+1)}\widetilde{q}},{_{(n+1)}\widetilde{q}(x,A)})$ on ${ }_{(n+1)}\widetilde{E}={ }_{(n)}E\cup\{{_{(n)}e}\}$,
which is given by
\begin{equation*}\label{nq-1}
{ }_{(n+1)}\widetilde{q}(x,A)=q(x,A\cap{{ }_{(n)}E})+\mathbf{1}_{\{{_{(n)}e}\in{A}\}}q(x,{{ }_{(n)}E^C}),\ \ x\in{{{ }_{(n)}E}},A\in{{ }_{(n+1)}\mathcal{B}},
\end{equation*}
\begin{equation*}\label{nq-3}
{ }_{(n+1)}\widetilde{q}({_{(n)}e},A)=\frac{\int_{{{ }_{(n)}E^C}}\pi(dx)q(x,A\cap{{ }_{(n)}E})}{\int_{{ }_{(n)}E^C}\pi(dx)},\ \ A\in{{ }_{(n+1)}\mathcal{B}},
\end{equation*}
and
\begin{equation*}\label{nQ-4}
{ }_{(n+1)}\widetilde{q}(x)={ }_{(n+1)}\widetilde{q}(x,{ }_{(n+1)}\widetilde{E}),\ \ x\in{{{ }_{(n)}E}},
\end{equation*}
where ${{ }_{(n+1)}\mathcal{B}}$ is the {$\sigma$}-algebra generated by ${{{ }_{(n+1)}\widetilde{E}}}$.
Then, let $\big\{{ }_{(n+1)}\widetilde{X}_{t}, t\geq 0\big\}$ be the jump process on ${ }_{(n+1)}\widetilde{E}$ corresponding to the $q$-pair $({_{(n+1)}\widetilde{q}},{_{(n+1)}\widetilde{q}(x,A)})$,
and ${ }_{(n+1)}\widetilde{\pi}$ be the stationary distribution of ${ }_{(n+1)}\widetilde{X}_{t}$ such that
${ }_{(n+1)}\widetilde{\pi}(A)=\int_{A\cap{{ }_{(n)}E}}\pi(dx)$ for ${_{(n)}e}\notin{A}$
and ${ }_{(n+1)}\widetilde{\pi}(A)=\int_{A\cap{{ }_{(n)}E}}\pi(dx)+\int_{{ }_{(n)}E^C}\pi(dx)$ for ${_{(n)}e}\in{A}$.
Finally, let ${_{(n+1)}\widetilde{g}}$ be a continuous truncation function satisfying ${ }_{(n+1)}\widetilde{g}(x)=g(x)$ for $x\in{_{(n)}E}$ and ${ }_{(n+1)}\widetilde{g}({_{(n)}e})=0$.

By following a similar proof technique as in Lemma \ref{3-1} and \ref{3-2}, we obtain the following convergence result.
\begin{lemma}\label{3-1=1}

For the collapsed process ${ }_{(n)}\widetilde{X}_{t}$, we have  that
\begin{description}
  \item[(i)] if $g$ is a bounded continuous function on $E$, then
  		$$
		\mathbb{P}_i\left(\omega\in \varOmega:\lim_{n \rightarrow \infty}{_{(n)}\widetilde{g}}\big({ }_{(n)}\widetilde{X}_{s}(\omega) \big)=g\left(X_{s}(\omega)\right) \right)=1,
		$$
and
		$$
		\lim_{n \rightarrow \infty} { }_{(n)}\widetilde{\pi}\left({ }_{(n)}\widetilde{g} \right)=\pi(g);
		$$	
  \item[(ii)] and if the original process $X_t$ admits an $\mathcal{L}^{2}(\pi)$-spectral gap $\lambda(q)$, and $\mathcal{K}$  is dense in $D^*$, then
  \[
	\lim_{n\rightarrow\infty}	{\lambda}({ }_{(n)}\widetilde{q}) = {\lambda}(q).
\]
\end{description}
	\end{lemma}	

By combining Lemmas \ref{lemma-1-ctmc} and \ref{3-1=1}, and using a similar argument in the proof of Theorem \ref{main result},
we obtain the following Hoeffding's inequality for a jump process on a general state space.

\begin{theorem}
Assume that the process $X_t$ exists an $\mathcal{L}^{2}(\pi)$-spectral gap $\lambda(q)$, and that $\mathcal{K}$ is dense in $D^*$.
Then for any bounded continuous function $g:E\to[a,b]$, and $t, \varepsilon >0$, we have
\[
		\mathbb{P}_{\pi}\left(\frac{1}{t} \int_{0}^{t} g\left(X_{s}\right)\mathrm{d}s-\pi (g) \geq \varepsilon \right) \leq \exp\left\{-\frac{\lambda(q)t\varepsilon^2}{(b-a)^2} \right\}.
\]
\end{theorem}

\end{document}